\documentclass[12pt]{amsart}
\usepackage{graphicx}
\usepackage{amssymb}
\usepackage{amscd}
\usepackage{citesort}
\newtheorem{tw}{Theorem}[section]
\newtheorem{prop}[tw]{Proposition}
\newtheorem{lem}[tw]{Lemma}
\newtheorem{wn}[tw]{Corollary}

\theoremstyle{remark}
\newtheorem{uw}[tw]{Remark}

\theoremstyle{definition}

\newcommand{\cal}[1]{\mathcal{#1}}
\newcommand{\bez}{\setminus}

\newcommand{\eps}{\varepsilon}
\newcommand{\ro}{\varrho}
\newcommand{\fal}[1]{\widetilde{#1}}

\newcommand{\kre}[1]{\overline{#1}}
\newcommand{\gen}[1]{\langle #1 \rangle}
\newcommand{\map}[3]{#1\colon #2\to #3}
\newcommand{\field}[1]{\mathbb{#1}}
\newcommand{\zz}{\field{Z}}
\newcommand{\rr}{\field{R}}

\newcommand{\st}{\;|\;}

\newcommand{\lst}[2]{{#1}_1,\dotsc,{#1}_{#2}}
\newcommand{\Mob}{M\"{o}bius strip}

\begin{document}

\numberwithin{equation}{section}

\title[The twist subgroup of the\ldots]
{The twist subgroup of the mapping class group of a nonorientable surface}

\author{Micha\l\ Stukow}

\thanks{Supported by the Foundation for Polish Science. Some part of the research was done while the author was visiting the Institute Mittag-Leffler,
Djursholm, Sweden. Their hospitality and financial support is greatly appreciated.}
\address[]{
Institute of Mathematics, University of Gda\'nsk, Wita Stwosza 57, 80-952 Gda\'nsk, Poland }

\email{trojkat@math.univ.gda.pl}



\begin{abstract}
Let ${\cal{T}}(N)$ be the subgroup of the mapping class group of a nonorientable surface $N$ (possibly
with punctures and/or boundary components) generated by twists about two--sided circles.  We obtain a simple generating set for
${\cal{T}}(N)$. As an application we compute the first homology group (abelianization) of ${\cal{T}}(N)$.
\end{abstract}

\maketitle%
\section{Introduction}%
Let $N_{g,s}^n$ be a smooth, nonorientable, compact surface of genus $g$ with $s$
boundary components and $n$ punctures. If $s$ and/or $n$ is zero then we omit it from the
notation. If we do not want to emphasise the numbers $g,s,n$, we simply write $N$ for a
surface $N_{g,s}^n$. Recall that $N_{g}$ is a connected sum of $g$ projective planes and
$N_{g,s}^n$ is obtained from $N_g$ by removing $s$ open disks and specifying the set
$\Sigma=\{\lst{z}{n}\}$ of $n$ distinguished points in the interior of $N_g$.

Let ${\textrm{Diff}}(N)$ be the group of all diffeomorphisms $\map{h}{N}{N}$ such that $h$ is the identity
on each boundary component and $h(\Sigma)=\Sigma$. By ${\cal{M}}(N)$ we denote the quotient group of
${\textrm{Diff}}(N)$ by the subgroup consisting of maps isotopic to the identity, where we assume that
isotopies fix $\Sigma$ and are the identity on each boundary component. ${\cal{M}}(N)$ is called the
\emph{mapping class group} of $N$. The mapping class group of an orientable surface is defined
analogously, but we consider only orientation preserving maps.

If we assume that maps and isotopies fix the set $\Sigma$ pointwise then we obtain the so--called
\emph{pure mapping class group} ${\cal{PM}}(N_{g,s}^n)$. If we also require that maps preserve the local
orientation around each of the punctures then we obtain the group ${\cal{PM}}^+(N_{g,s}^n)$. It is an easy
observation that ${\cal{PM}}(N_{g,s}^n)$ is the subgroup of index $n!$ in ${\cal{M}}(N_{g,s}^n)$ and
${\cal{PM}}^+(N_{g,s}^n)$ is of index $2^n$ in ${\cal{PM}}(N_{g,s}^n)$.

Define also ${\cal{T}}(N)$ to be the \emph{twist subgroup} of ${\cal{M}}(N)$, that is the subgroup generated
by Dehn twists about two--sided circles.

By abuse of notation we will use the same letter for a map and its isotopy class and we
will use the functional notation for the composition of diffeomorphisms.

%
%
%
\subsection{Background}
The study of algebraic properties of mapping class groups of an orientable surface goes back to the work of Dehn \cite{Dehn} and Nielsen \cite{Nil1,Nil2,Nil3}. Probably the best modern exposition of this reach theory is a survey article by Ivanov \cite{IvanovSurv}. On the other hand, the nonorientable case has not been studied much. The first significant result is due
to Lickorish \cite{Lick3}, who proved that the twist subgroup ${\cal{T}}(N_g)$ is a subgroup of index $2$
in the mapping class group ${\cal{M}}(N_{g})$. Moreover, ${\cal{M}}(N_{g})$ is generated by Dehn twists
and a so--called ``crosscap slide" (or a ``Y--homeomorphism"). Later Chillingworth \cite{Chil} found a finite generating set for the group ${\cal{M}}(N_{g})$, and Birman and Chillingworth \cite{BirChil1} showed how this generating set can be derived from the known properties of the mapping class group of the orientable double cover of~$N_g$.

These studies were continued much later by Korkmaz \cite{Kork-non}, who found finite generating sets for the groups ${\cal{M}}(N_{g}^n)$ and ${\cal{PM}}(N_{g}^n)$. Korkmaz \cite{Kork-non1,Kork-non} also computed the first integral homology group of ${\cal{M}}(N_{g}^n)$, and under additional assumption $g\geq 7$, of ${\cal{PM}}(N_{g}^n)$. He also showed \cite{Kork-non} that ${\cal{T}}(N_{g}^n)$ is of index $2^{n+1}n!$ in ${\cal{M}}(N_{g}^n)$, provided $g\geq 7$.

Recently \cite{Stukow_SurBg}, we extended some of the above results to arbitrary
mapping class groups ${\cal{M}}(N_{g,s}^n)$ and ${\cal{PM}}(N_{g,s}^n)$, provided $g\geq 3$. In particular
we obtained simple generating sets for these groups and we computed their abelianizations.

Finally, let us mention that recently Wahl \cite{Wahl_stab} proved some stability theorems for the homology of mapping class groups of nonorientable surfaces and using the ideas of Madsen and Weiss \cite{MadWeiss} she managed to identify the stable rational cohomology of ${\cal{M}(N)}$.

Another very promising project is the work of Szepietowski \cite{Szep_curv} who showed a method to obtain a presentation for the group  ${\cal{M}}(N_{g,s}^n)$. Using this technique he managed \cite{Szep_gen4} to derive a presentation of the group ${\cal{M}}(N_{4})$.

\subsection{Main results}
The following paper is a natural continuation of the results mentioned above, especially it should be thought as a continuation of \cite{Stukow_SurBg}. Namely, we study basic
algebraic properties of the twist subgroup of the mapping class group of a nonorientable surface of genus
$g\geq3$. The crucial observation which makes such a study possible is that ${\cal{T}}(N_{g,s}^n)$ is a
subgroup of index $2$ in ${\cal{PM}}^+(N_{g,s}^n)$ (hence of index $2^{n+1}n!$ in ${\cal{M}}(N_{g,s}^n)$
-- see Corollaries \ref{wn:T:eq:falT} and \ref{wn:T:eq:falT2}). Using this observation, we obtain
surprisingly simple generating set for the twist subgroup -- see Theorem \ref{tw:gen:T}. Moreover, we
compute the first integral homology group (abelianization) of this subgroup -- see Theorem
\ref{tw:main:hom}.

\section{Preliminaries}
By a \emph{circle} on $N$ we mean an unoriented simple closed curve on $N\bez \Sigma$, which is
disjoint from the boundary of $N$. Usually we identify a circle with its image. Moreover, as in the case of diffeomorphisms, we will use the same letter for a circle and its isotopy class. According to whether a regular neighbourhood of
a circle is an annulus or a \Mob, we call the circle \emph{two--sided} or \emph{one--sided} respectively.
We say that a circle is \emph{generic} if it bounds neither a disk with less than two punctures nor a \Mob\ disjoint from $\Sigma$.

Let $a$ be a two--sided circle. By definition, a regular neighbourhood of $a$ is an
annulus, so if we fix one of its two possible orientations, we can define the
\emph{right Dehn twist} $t_a$ about $a$ in the usual way. We emphasise that since we are
dealing with nonorientable surfaces, there is no canonical way to choose the orientation
of $S_a$. Therefore by a twist about $a$ we always mean one of two possible twists about
$a$ (the second one is then its inverse). By a \emph{boundary twist} we mean a twist
about a circle isotopic to a boundary component. It is known that if $a$ is not
generic then the Dehn twist $t_a$ is trivial. In particular a Dehn twist about the
boundary of a \Mob\ is trivial -- see Theorem 3.4 of \cite{Epstein}.

Other important examples of diffeomorphisms of a nonorientable surface are the
\emph{crosscap slide} and the \emph{puncture slide}. They are defined as a slide of a
crosscap and of a puncture respectively, along a loop. The general convention is that one
considers only crosscap slides along one--sided simple loops (in such a form it was
introduced by Lickorish \cite{Lick3}; for precise definitions and properties see
\cite{Kork-non}).

The following two propositions follow immediately from the above definitions.

\begin{prop}\label{pre:prop:twist:cong}
Let $N_a$ be an oriented regular neighbourhood of a two--sided circle $a$ in a surface
$N$, and let $\map{f}{N}{N}$ be any diffeomorphism. Then $ft_af^{-1}=t_{f(a)}$, where the
orientation of a regular neighbourhood of $f(a)$ is induced by the orientation of
$f(N_a)$.\qed
\end{prop}

\begin{prop}\label{pre:prop:Yhomeo:cong}
Let $K$ be a subsurface of $N$ which is a Klein bottle with one boundary component $\xi$, and let $y$ be a crosscap slide on $K$ such that $y^2=t_\xi$. Then for any diffeomorphism $\map{f}{N}{N}$, $fyf^{-1}$ is a crosscap slide on $f(K)$ such that
\[(fyf^{-1})^2=t_{f(\xi)},\]
where the orientation of a regular neighbourhood of $f(\xi)$ is induced via $f$ by the orientation of a regular neighbourhood of $\xi$.\qed
\end{prop}

%

One of our main tools in studying properties of mapping class groups is the so--called \emph{lantern relation}. The proof can be found in Section 4
of \cite{John1}.
%
%
\begin{prop}
Let $S$ be a sphere with four holes embedded in a surface $N\bez \Sigma$ and let
$a_0,a_1,a_2,a_3$ be the boundary circles of $S$. Define also $a_{1,2},a_{2,3},a_{1,3}$
as in Figure \ref{01_2_pre}(i) and assume that the orientations of regular neighbourhoods of
these seven circles are induced from the orientation of $S$.
Then \[t_{a_0}t_{a_1}t_{a_2}t_{a_3}=t_{a_{1,2}}t_{a_{2,3}}t_{a_{1,3}}.\] \qed
\end{prop}
\begin{figure}[h] \begin{center}
\includegraphics{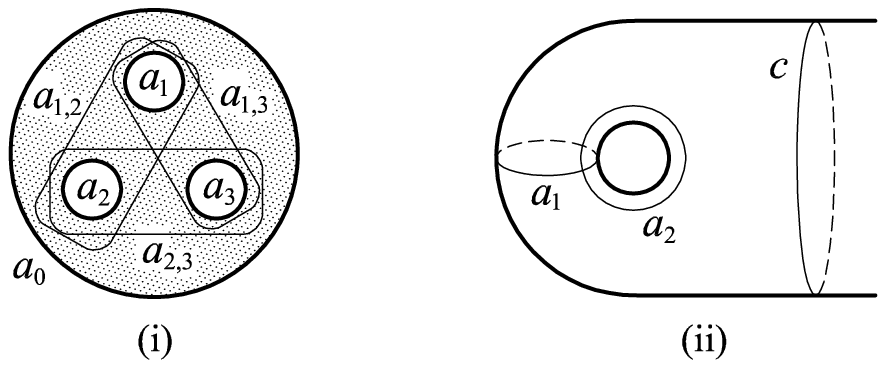}
\caption{Circles of the lantern and torus with a hole relations.}\label{01_2_pre}
\end{center} \end{figure}
Finally, recall the \emph{torus with a hole relation} -- see Lemma 3 of \cite{Lick2}. It can be also thought as an instance of the so--called \emph{star} relation \cite{Gervais_top}.
\begin{prop}\label{prop:tor:hole}
Let $S$ be a torus with one boundary component $c$ embedded in a surface $N\bez \Sigma$ and let $a_1$ and $a_2$ be two two--sided circles as in Figure \ref{01_2_pre}(ii).
If the orientations of regular neighbourhoods of $a_1,a_2$ and $c$ are induced from the orientation of $S$ then
\[(t_{a_1}t_{a_2})^6=t_c.\]\qed
\end{prop} 
\section{Two models for a nonorientable surface}\label{sec:mod2}
Let $g=2r+1$ for $g$ odd and $g=2r+2$ for $g$ even. Represent the surface
$N=N_{g,s}^n$ as a connected sum of an orientable surface of genus $r$ and one or two
projective planes (one for $g$ odd and two for $g$ even). Figures \ref{fig:010_gen_odd}
and \ref{fig:020_gen_even} show this model of $N$ -- in these figures the big shaded
disks represent crosscaps, hence their interiors are to be removed and then the antipodal
points on each boundary component are to be identified. The smaller shaded disks
represent components of $\partial N$ (we will call them \emph{holes}).
\begin{figure}[h]
\includegraphics{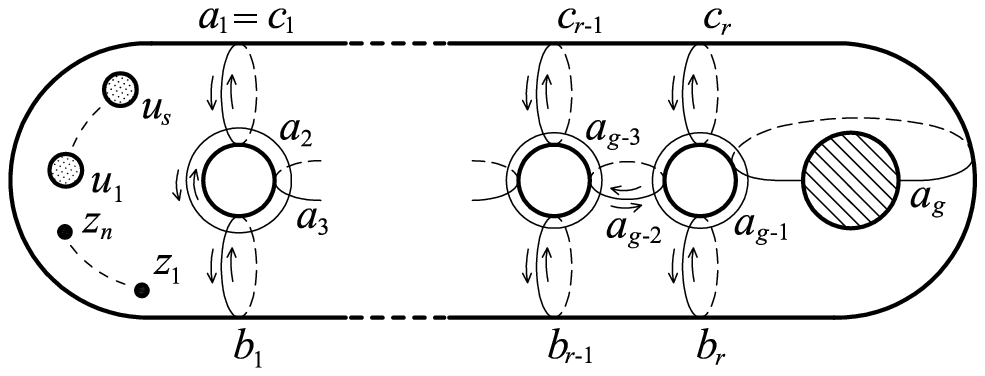}
\caption{Circles $a_i,b_i,c_i$ and $u_i$ for $g=2r+1$.}\label{fig:010_gen_odd}
\end{figure}
\begin{figure}[h]
\includegraphics{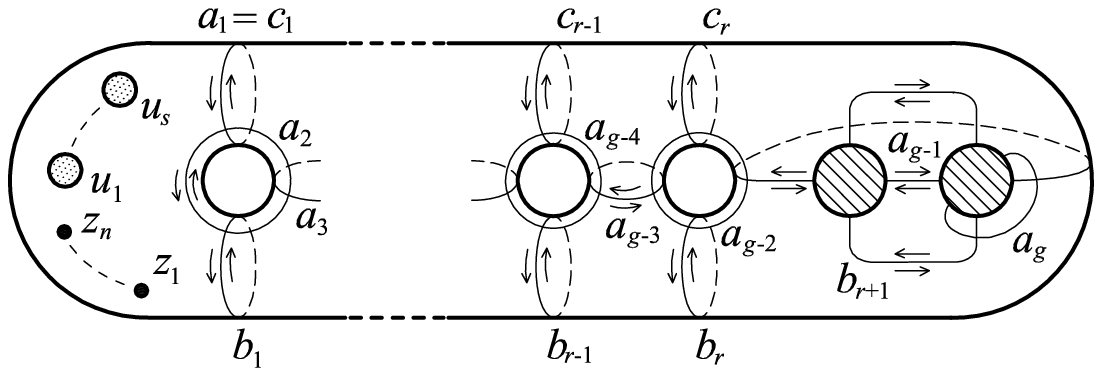}
\caption{Circles $a_i,b_i,c_i$ and $u_i$ for $g=2r+2$.}\label{fig:020_gen_even}
\end{figure}

It is well known that $N=N_{g,s}^n$ can be also represented as a sphere with $n$
punctures, $s$ holes and $g$ crosscaps -- cf Figure \ref{060_mod2}.
\begin{figure}[h]
\includegraphics{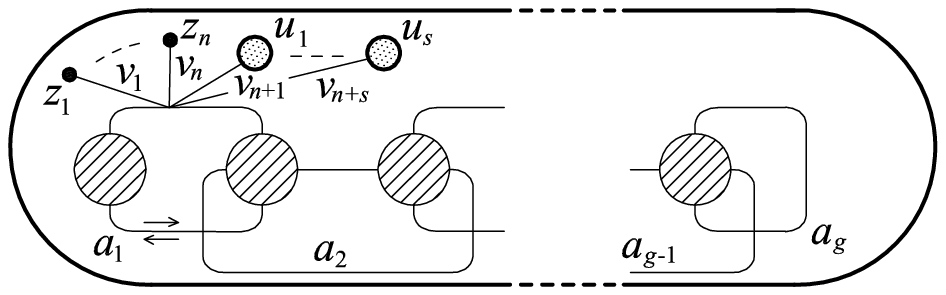}
\caption{Surface $N$ as sphere with crosscaps.}\label{060_mod2}
\end{figure}
In order to distinguish this model from the previous one (provided by Figures
\ref{fig:010_gen_odd} and \ref{fig:020_gen_even}), let us denote it by $\fal{N}$. The
goal of this section is to construct an explicit homeomorphism $\map{\Phi}{N}{\fal{N}}$.

Let $\lst{a}{g}$ and $\lst{u}{s}$ be two--sided circles on $N$ as in Figures \ref{fig:010_gen_odd}
and \ref{fig:020_gen_even} ($u_i$ are the boundary circles).
\begin{figure}[h]
\includegraphics{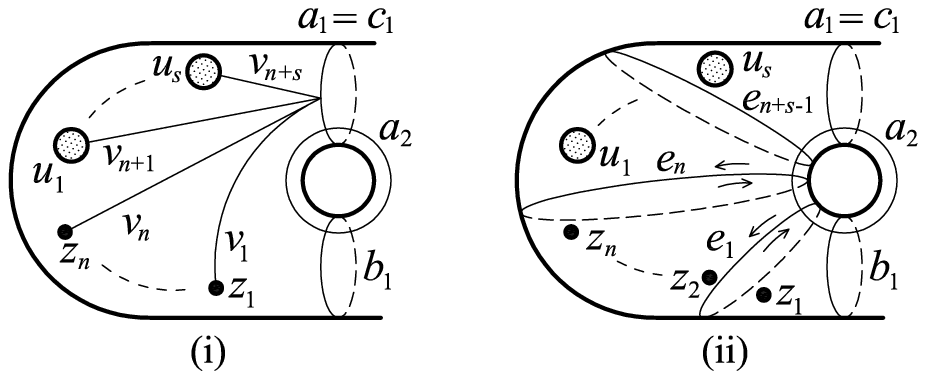}
\caption{Circles $e_i$ and arcs $v_i$.}\label{fig:030_gen_ef}
\end{figure}
Define also $\lst{v}{n+s}$ to be the arcs as in Figure \ref{fig:030_gen_ef}(i) and
observe that if we cut $N$ along the circles $\lst{a}{g}$ and arcs $\lst{v}{n+s}$ we
obtain a polygon $\Delta$ with sides
\begin{multline*}a_1,\ldots,
a_{g-1},a_g,a_{g-1},\ldots, a_1,v_1,v_1,\ldots, v_n,v_n,\\
v_{n+1},u_1,v_{n+1},\ldots, v_{n+s},u_s,v_{n+s}, a_1,\ldots, a_{g-1},a_g,a_{g-1},\ldots,a_2
\end{multline*}
where the labels indicate from which circle/arc the edge came from. Identical polygon
can be obtained by cutting the surface $\fal{N}$ along the circles $\lst{a}{g}$ and
arcs $\lst{v}{n+s}$ indicated in Figure \ref{060_mod2}. Moreover, it is not difficult to
see that the identification patterns required to reconstruct $N$ and $\fal{N}$ form
$\Delta$ are identical (cf Figures \ref{070_cut_odd} and \ref{080_cut_even}).
\begin{figure}[h]
\includegraphics[scale=1.1]{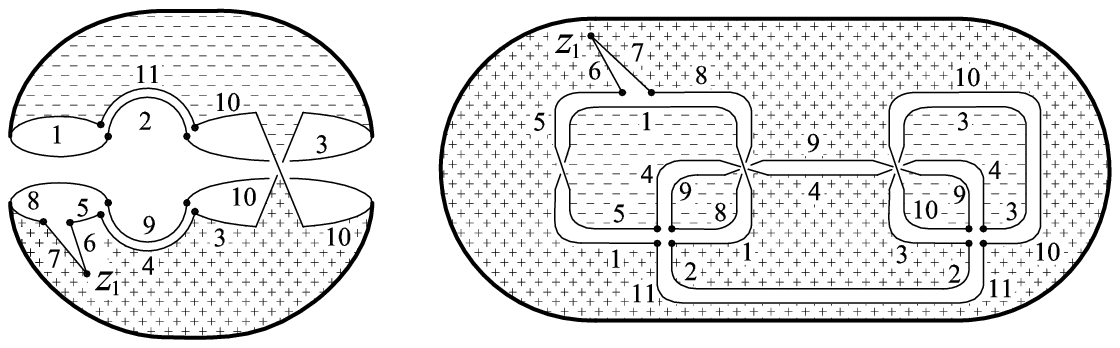}
\caption{Cutting $N$ and $\fal{N}$ for $(g,s,n)=(3,0,1)$.}\label{070_cut_odd}
\end{figure}
\begin{figure}[h]
\includegraphics[scale=1.1]{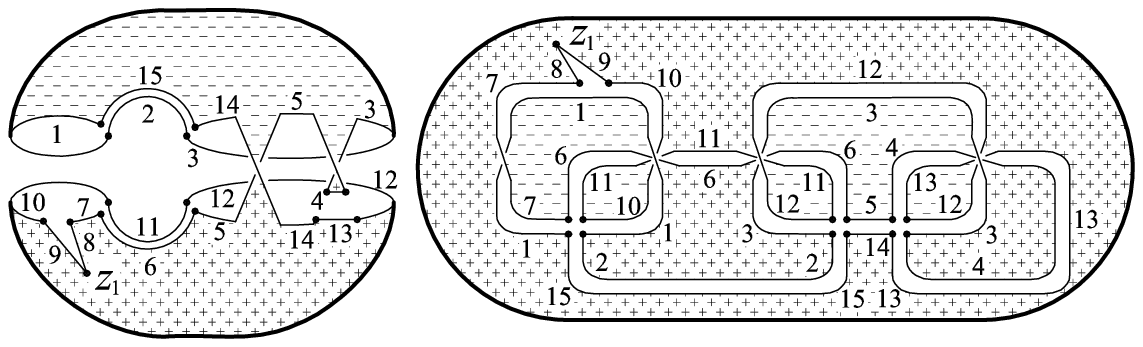}
\caption{Cutting $N$ and $\fal{N}$ for $(g,s,n)=(4,0,1)$.}\label{080_cut_even}
\end{figure}
This implies that there exists a homeomorphism $\map{\Phi}{N}{\fal{N}}$ which maps
circles $\lst{a}{g}$, $\lst{u}{s}$ and arcs $\lst{v}{n+s}$ in Figures
\ref{fig:010_gen_odd}, \ref{fig:020_gen_even}, \ref{fig:030_gen_ef}(ii) to the
circles/arcs with the same labels in Figure \ref{060_mod2}.

The above geometric description of $\Phi$ is very convenient because it provides a
simple method for transferring circles between two models of $N$. In fact, if $c$ is a
circle on $N$ then $c$ becomes a collection of arcs in $\Delta$. Moreover, up to isotopy
we can assume that $c$ does not pass through any of the vertices of $\Delta$. Since
$\Delta$ is simply connected each of this arcs is uniquely determined by the position of
its endpoints. To obtain the image $\Phi(c)$ it is enough to reconstruct the surface
$\fal{N}$ from $\Delta$ keeping track of the collection of arcs composing $c$. In
practise this can be easily done using pictures like Figures \ref{070_cut_odd} and
\ref{080_cut_even}. Moreover, it is not difficult to see that we can transfer not only the circles but also the orientations of their neighbourhoods (if it exists) -- small plus and minus signs in Figures \ref{070_cut_odd} and \ref{080_cut_even} indicate our choice of the orientation of $\Delta$. Of course the above procedure works as well in the other direction.

Keeping in mind the above description, from now on we will transfer circles form $N$ to
$\fal{N}$ and vice versa without further comments.

\section{Generators for the group ${\cal{PM}^+(N_{g,s}^n)}$}

Let
\[
{\cal{C}}=\{a_2,\ldots,{a}_{g-1},\lst{b}{r},\lst{c}{r},\lst{e}{n+s-1},\lst{u}{s}\}\]
for $g$ odd, and
\[
{\cal{C}}=\{a_2,\ldots,{a}_{g-1},\lst{b}{r+1},\lst{c}{r},\lst{e}{n+s-1},\lst{u}{s}\}\]
for $g$ even,
where the circles $a_i,b_i,c_i,u_i$ are as in Figures \ref{fig:010_gen_odd} and
\ref{fig:020_gen_even} and $e_i$ are as in Figure \ref{fig:030_gen_ef}(ii). Moreover,
these figures indicate the orientations of local neighbourhoods of circles in $\cal{C}$.
We did it by indicating the direction of twists about these circles. Therefore by a twist
about one of the circles in $\cal{C}$ we will always mean the twist determined by this
particular choice of orientation (the general rule is that we consider \emph{right Dehn
twists}, that is if we approach the circle of twisting we turn to the right.
%
%
Define also $y$ to be a crosscap slide supported on a Klein bottle cut of by the circle
$\xi$ indicated in Figure \ref{fig:050_gen_KsiLam}.
\begin{figure}[h]
\includegraphics{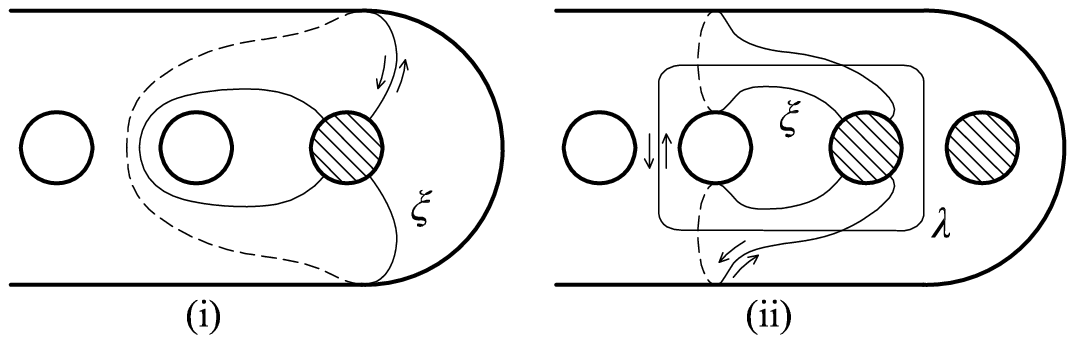}
\caption{Circles $\xi$ and $\lambda$.}\label{fig:050_gen_KsiLam}
\end{figure}
To be more precise, in terms of the model $\fal{N}$, let $C_{g-1}$ and $C_g$ be crosscaps
as in Figure~\ref{fig:090_act_y}.
\begin{figure}[h]
\includegraphics{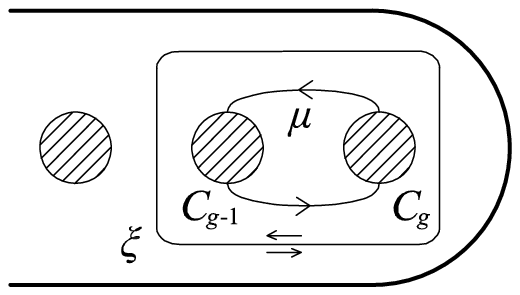}
\caption{Circle $\xi$.}\label{fig:090_act_y}
\end{figure}
The same figure shows the circle $\xi$ -- it cuts off these two crosscaps. Hence we can
define $y$ to be a slide of the crosscap $C_{g}$ along the path $\mu$ indicated in that
figure. In particular $y^2=t_\xi$.

Now we are ready to state the main theorem of this section, which is a simplification of
the known generating set for the group ${\cal{PM}^+(N_{g,s}^n)}$.

\begin{tw}\label{tw:gen:pure}
Let $g \geq 3$. Then the mapping class group ${\cal{PM}^+(N_{g,s}^n)}$ is generated by
$  \{t_l,y\st l\in{\cal{C}} \}$.
\end{tw}
\begin{proof}
By Theorem 5.2 of \cite{Stukow_SurBg} and by Propositions \ref{pre:prop:twist:cong} and \ref{pre:prop:Yhomeo:cong}, the group ${\cal{PM}^+(N_{g,s}^n)}$ is generated by
\begin{itemize}
 \item $\{t_l,y\st l\in{\cal{C'}} \}$ if $g$ is odd
 \item $\{t_l,(t_{a_{g-2}}t_{a_{g-1}})^{-1}y(t_{a_{g-2}}t_{a_{g-1}}),t_{b_r}^{-1}t_\lambda t_{b_r}\st l\in{\cal{C'}}\}$ if $g$ is even,
\end{itemize}
where $\lambda$ is as in Figure~\ref{fig:050_gen_KsiLam}(ii) and
\[{\cal{C'}}={\cal{C}}\cup \{ f_1,\ldots,f_{n+s}\} \]
for
$\lst{f}{n+s}$ as in Figure \ref{fig:040_gen_ff}(i)
(Figure \ref{fig:040_gen_ff}(i) defines $f_i$ for $g$ even; for $g$ odd just forget about
the second crosscap).
\begin{figure}[h]
\includegraphics{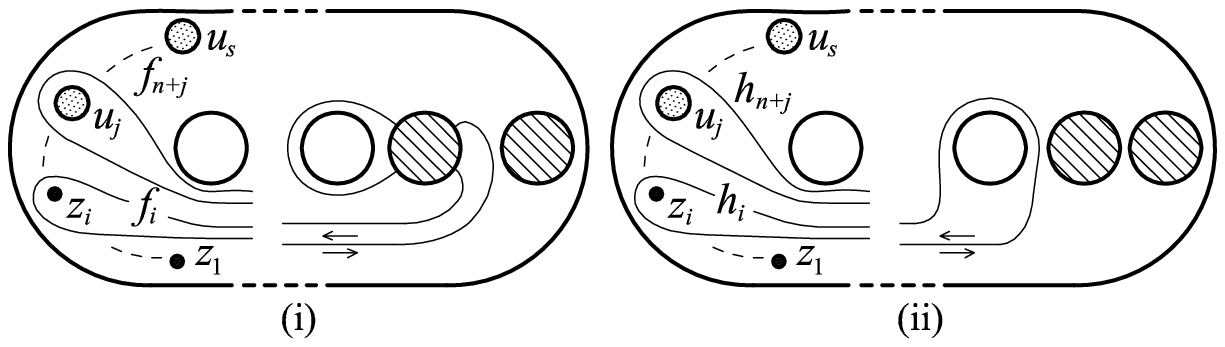}
\caption{Circles $f_i$ and $h_i$.}\label{fig:040_gen_ff}
\end{figure}

Therefore to complete the proof it is enough to show that if $G=\gen{t_l,y\st l\in{\cal{C}}}$ then $t_{f_1},\ldots,t_{f_{n+s}}\in G$, and $t_\lambda\in G$ for $g$ even.

Let $\lst{h}{n+s}$ be circles as in Figure \ref{fig:040_gen_ff}(ii) (as before, for $g$ odd forget
about the second crosscap).
We claim that
\begin{itemize}
\item $f_i=t_{c_{r}}^{-1}t_{a_{g-1}}^{-2}y^{-1}t_{b_{r}}(h_i)$ if $g$ is odd,
\item $f_i=y^{-1}t_{b_{r+1}}^{-1}t_{a_{g-1}}^{-1}t_\lambda t_{a_{g-2}}^{-1}t_{a_{g-1}}^{-1}t_{b_{r}}t_{a_{g-2}}^2t_{b_{r}}(h_i)$ if $g$ is even.
\end{itemize}
In fact, using the procedure of transferring circles between two models $N$ and $\fal{N}$
described in Section \ref{sec:mod2}, it is not difficult to check that Figure~\ref{fig:095_h_i}(i) shows the circle $t_{b_r}(h_i)$ on $\fal{N}$ and Figure~\ref{fig:095_h_i}(ii) shows $y^{-1}(t_{b_r}(h_i))$. Then transferring this circle back to $N$ easily leads to the first of the above relations. The second one can be proved analogously.
\begin{figure}[h]
\includegraphics{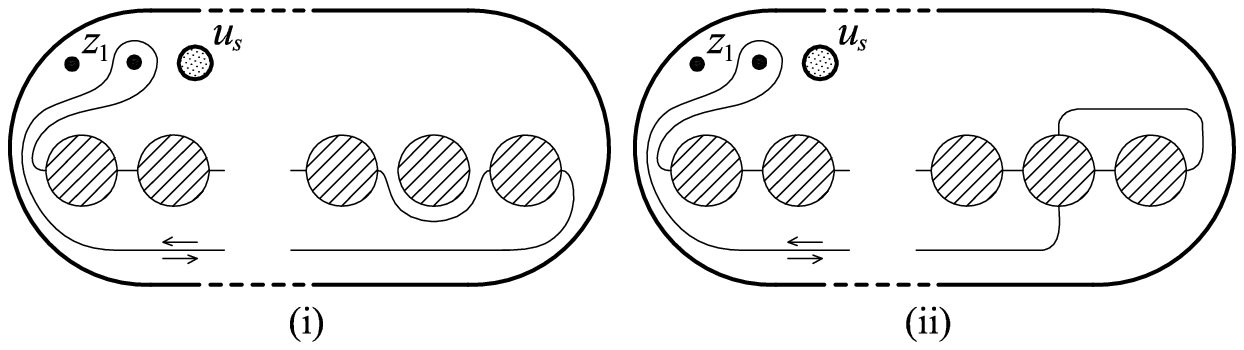}
\caption{Circles $t_{b_r}(h_i)$ and $y^{-1}t_{b_r}(h_i)$ for $g$ odd.}\label{fig:095_h_i}
\end{figure}

By Lemma 3.3 of \cite{Stukow_SurBg}, $t_{h_i}\in G$ for $i=1,\ldots,n+s$. Hence by Proposition
\ref{pre:prop:twist:cong}, $t_{f_i}\in G$ provided $t_\lambda\in G$. But this follows from
the relation
\[\lambda=t_{b_{r+1}}y (a_{g-2}). \]
%
\end{proof}

\section{The action of $\cal{M}(N_{g,s}^n)$ on $\mathrm{H_1}(N_{g,s}^n,\rr)$}\label{sec:act}%
It is known that $\mathrm{H_1}(N,\rr)$ has a basis in which every linear map $\map{f_*}{\mathrm{H_1}(N,\rr)}{\mathrm{H_1}(N,\rr)}$, induced by a diffeomorphism $\map{f}{N}{N}$,
has a matrix with integral coefficients. Therefore we can define the
\emph{determinant homomorphism} $\map{D}{{\cal{PM}}^+(N)}{\zz_2}$ as follows: $D(f)=\det(f_*)$
(we use the multiplicative notation for the group $\zz_2$).
\begin{lem}\label{lem:ac:tw:nonsep}
Let $c$ be a two--sided nonseparating circle on $N$. Then $D(t_c)=1$.
\end{lem}
\begin{proof}
It is an easy topological fact, that if $c_1$ and $c_2$ are two--sided, nonseparating
circles in $N$ such that both $N\bez c_1$ and $N\bez c_2$ are either orientable or
nonorientable, then $f(c_1)=c_2$ for some diffeomorphism of $N$. Let us stress the fact
that $f$ may not be the identity on $\partial N$, hence we can not assume that $f$
induces an element of the group $\cal{M}(N)$, however this is of no importance to what follows. In particular, $(t_{c_1})_*$ and
$(t_{c_2})_*$ are conjugate in ${\mathrm{Aut}}({\mathrm{H_1}(N,\rr)})$, hence $D(t_{c_1})=D(t_{c_2})$.
Moreover, if $g$ is odd then there is no nonseparating two--sided circle $c$ on $N$ such that $N\bez
c$ is orientable. Therefore, to prove the lemma it is enough to show that
$D(t_{a_1})=1$ and that $D(t_{b_{r+1}})=1$ if $g$ is even. This can be easily done -- we
skip the computations.
\end{proof}
\begin{lem} \label{lem:ac:tw:sep}
Let $c$ be a two--sided separating circle on $N$. Then
$\map{(t_c)_*}{\mathrm{H_1}(N,\rr)}{\mathrm{H_1}(N,\rr)}$ is the identity map.
\end{lem}
\begin{proof}
Since $c$ is two--sided, we can fix an orientation of a regular neighbourhood of $c$.
Therefore for any circle $a$ which is transversal to $c$, we can define the algebraic
intersection number $I(c,a)$ in a usual way (we do not claim that $I(c,a)$ has any
particular properties -- it just counts the points $c\cap a$ with appropriate signs). By
the definition of a twist, it is obvious that
\[[t_c(a)]=[a]\pm I(c,a)[c]\]
where $[x]$ denotes the homology class of $x$. Moreover, it is clear that if $c$ is
separating then $I(c,a)=0$ for any circle $a$. Therefore
\[[t_c(a)]=[a]\] and the lemma follows by the fact that homology classes of circles span
${\mathrm{H_1}(N,\rr)}$.
\end{proof}
By Lemmas \ref{lem:ac:tw:nonsep} and \ref{lem:ac:tw:sep} we obtain
\begin{prop}\label{prop:det:tw}
Let $c$ be a two--sided circle on $N$. Then \mbox{$D(t_c)=1$}.\qed
\end{prop}
The explicit definition of $y$ made in the previous section easily leads to the following
\begin{prop}\label{lem:ac:cross}
Let $\map{D}{{\cal{PM}^+(N)}}{\zz_2}$ as above. Then \mbox{$D(y)=-1$}.\qed
\end{prop}

\section{Generators for the twist subgroup}
The main goal of this section is to find a simple generating set for the twist subgroup ${\cal{T}}(N)$.
Our main tool will be the following well known fact from combinatorial group theory -- see for example Chapter 9 of~\cite{Johnson}.
\begin{prop}\label{prop:Joh}
Let $X$ be a generating set for a group $G$ and let $U$ be a left transversal for a
subgroup $H$. Then $H$ is generated by the set
\[\{ux\kre{ux}^{-1}\, :\, u\in U, x\in X, ux\not\in U\},\]
where $\kre{g}=gH\cap U$ for $g\in G$. \qed\end{prop}

Let ${\fal{\cal{T}}}(N)$ denote the kernel of the homomorphism \[\map{D}{{\cal{PM}^+(N)}}{\zz_2}\] defined
in Section \ref{sec:act}. The reason for our choice of notation will become apparent after Corollary
\ref{wn:T:eq:falT} below, where we will prove that in fact $\fal{\cal{T}}(N)={\cal{T}}(N)$ is the twist
subgroup.
\begin{tw}\label{tw:gen:T}
Let $g\geq 3$. Then the group ${\fal{\cal{T}}}(N)$ is generated by
\begin{itemize}
 \item $\{t_l\st l\in {\cal{C}}\cup\{f_1,\ldots,f_{n+s-1},\xi\} \}$ if $g=3$,
 \item $\{t_l\st l\in {\cal{C}}\cup\{\psi,\xi \} \}$ if $g\geq 5$ odd,
 \item $\{t_l\st l\in{\cal{C}}\cup\{\lambda,\psi,\xi\} \}$ if $g$ even,
\end{itemize}
where $\psi$ is the two--sided circle indicated in Figure \ref{fig:100_genT_Psi}, $\lst{f}{n+s-1}$ are
as in Figure \ref{fig:040_gen_ff}(i), $\lambda$ is as in Figure \ref{fig:050_gen_KsiLam}(ii)
and $\xi$ is as in Figure \ref{fig:050_gen_KsiLam}.
\end{tw}
\begin{figure}[h]
\includegraphics{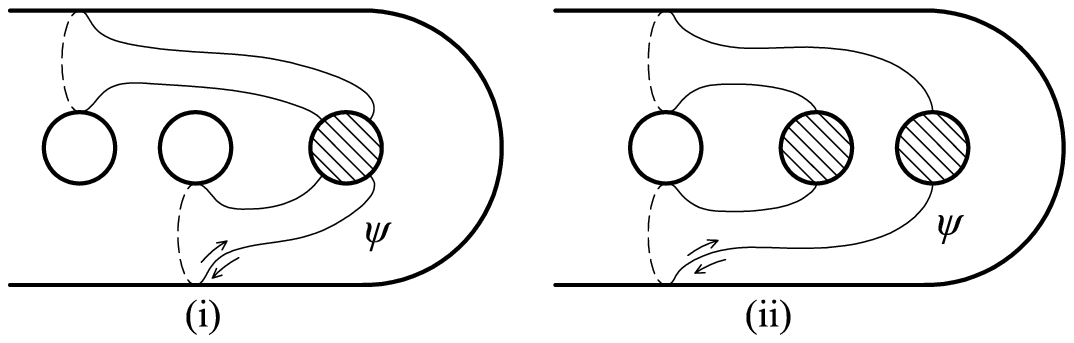}
\caption{Circle $\psi$ for $g$ odd and even.}\label{fig:100_genT_Psi}
\end{figure}
\begin{proof}
By Proposition \ref{lem:ac:cross}, $D$ is onto, hence
\[[{{\cal{PM}^+(N)}}:{\fal{\cal{T}}}(N)]=2\]
and as a transversal for ${\fal{\cal{T}}}(N)$ we can take $U=\{1,y\}$. By Theorem \ref{tw:gen:pure} and
Propositions \ref{prop:Joh}, \ref{prop:det:tw}, \ref{lem:ac:cross}, ${\fal{\cal{T}}}(N)$ is generated by
 \[\{t_l,yt_ly^{-1},y^2\st l\in {\cal{C}} \}.\]
Clearly the circles
\begin{itemize}
 \item $\lst{u}{s}$ if $g=3$,
 \item $a_2,\ldots,a_{g-3},\lst{b}{r-1},\lst{c}{r-1},\lst{e}{n+s-1},\\ \lst{u}{s}$ if $g\geq 5$   odd,
 \item $a_2,\ldots,a_{g-3},\lst{b}{r},\lst{c}{r},\lst{e}{n+s-1},\\ \lst{u}{s}$ if $g$ even,
\end{itemize}
are disjoint from the support of $y$ (cf Figures \ref{fig:010_gen_odd}, \ref{fig:020_gen_even} and
\ref{fig:050_gen_KsiLam}). Since $yt_ly^{-1}=t_{y(l)}^\pm$ and $y^2=t_\xi$, ${\fal{\cal{T}}}(N)$ is generated
by
\begin{itemize}
 \item $\{t_l\st l\in{\cal{C}}\cup\{y(a_2),y(b_1),y(c_1),y(e_1),\ldots,y(e_{n+s-1}),\xi \} \}$ if $g=3$,
 \item $\{t_l\st l\in{\cal{C}}\cup\{y(a_{g-2}),y(a_{g-1}),y(b_r),y(c_r),\xi\} \}$ if $g\geq5$ odd,
 \item $\{t_l\st l\in{\cal{C}}\cup\{y(a_{g-2}),y(a_{g-1}),y(b_{r+1}),\xi\} \}$ if $g$ even.
\end{itemize}

We defined $y$ using $\fal{N}$, so it is convenient to transfer circles to that model. Once this is done,
it is not difficult to check that we have relations:
\begin{itemize}
 \item $y(a_2)=a_2$, $y({c_1})={b_1}$, $y^{-1}({e_1})=\tau_1,\ldots,y^{-1}({e_{n+s-1}})=\tau_{n+s-1}$
if $g=3$, where $\tau_i$ is as in Figure \ref{fig:110_genT_tau},
 \item $y(a_{g-2})=\psi$, $y(a_{g-1})=a_{g-1}$, $y({c_r})={b_r}$ if $g\geq 5$ odd,
 \item $y(a_{g-2})=t_{b_{r+1}}^{-1}(\lambda)$, $y(a_{g-1})=a_{g-1}$, $y^{-1}(b_{r+1})=\psi$ if $g$ even.
\end{itemize}
\begin{figure}[h]
\includegraphics{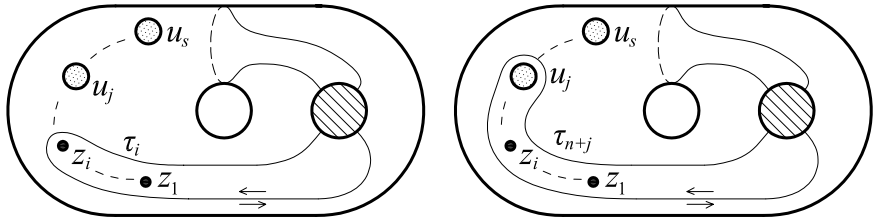}
\caption{Circles $\tau_1,\ldots,\tau_{n+s-1}$ for $g=3$.}\label{fig:110_genT_tau}
\end{figure}
To complete the proof it is enough to show that the twists
$t_{\tau_1},\ldots,t_{\tau_{n+s-1}}$ can be replaced by twists
$t_{f_1},\ldots,t_{f_{n+s-1}}$ in case $g=3$. In order to prove this, let
$\ro_i=t_{c_1}^{-1}t_{a_2}^{-1}(\tau_i)$. Then it is straightforward to check that
\[\tau_i=t_{\ro_{i-1}}^{-1}t_{a_2}^{2}t_{c_1}(f_i)\quad\text{for $i=2,\ldots,n+s-1$}. \]
Therefore, using the relation $\tau_1=f_1$, we can inductively replace each $t_{\tau_i}$
by $t_{f_i}$.
\end{proof}
\begin{wn}\label{wn:T:eq:falT}
Let $g\geq 3$. Then the kernel of the determinant homomorphism $\map{D}{{\cal{PM}^+(N)}}{\zz_2}$ is the
twist subgroup of $\cal{M}(N)$, that is $\fal{\cal{T}}(N)={\cal{T}}(N)$.
\end{wn}
\begin{proof}
Clearly ${\cal{T}}(N)\leq\cal{PM}^+(N)$ and by Proposition \ref{prop:det:tw}, \[{{\cal{T}(N)}\leq \ker
D=\fal{\cal{T}}(N)}.\] On the other hand, by Theorem \ref{tw:gen:T}, $\fal{\cal{T}}(N)\leq {\cal{T}}(N)$.
\end{proof}
\begin{wn}\label{wn:T:eq:falT2}
The twist subgroup ${\cal{T}}(N)$ has index $2^{n+1}n!$ in $\cal{M}(N_{g,s}^n)$.
\end{wn}
\begin{proof}
By the previous corollary, ${\cal{T}}(N)=\ker D$, so the conclusion follows from the obvious equality
\begin{multline*} [\cal{M}(N):\ker D]=[\cal{M}(N):\cal{PM}(N)]\cdot[\cal{PM}(N):\cal{PM}^+(N)]\cdot \\
\cdot[\cal{PM}^+(N):\ker D]=n!\cdot 2^n\cdot 2.
\end{multline*}
\end{proof}
In the case of a closed surface $N_{g}$, the above corollary was first proved by
Lickorish \cite{Lick3}. Later Korkmaz \cite{Kork-non} proved it for a punctured surface
$N_{g}^n$ under additional assumption $g\geq 7$.

\section{Homological results for the twist subgroup}
For the rest of the paper, for any $f\in {\cal{T}}(N)$ let $[f]$ denotes the homology
class of $f$ in $\mathrm{H_1}({\cal{T}}(N),\zz)$. Moreover, we will use the additive
notation in $\mathrm{H_1}({\cal{T}}(N),\zz)$.
\subsection{Homology classes of non--peripheral twists}
\begin{lem}\label{lem:cong:ess}
Let $a$ and $b$ be two two--sided circles on $N$ such that  $I(a,b)=1$ and the orientations of regular neighbourhoods of $a$ and $b$ are induced from the orientation of $a\cup b$.
Then $t_a$ and $t_b$ are conjugate in ${\cal{T}}(N)$. In particular $[t_{a}]=[t_{b}]$ in
$\mathrm{H_1}({\cal{T}}(N))$.
\end{lem}
\begin{proof}
The Lemma follows form the braid relation
\[t_b=(t_at_b)t_a(t_at_b)^{-1}. \]
\end{proof}
\begin{lem}\label{lem:two:gener:cong}
Assume that $g\geq 3$ and let $a$ and $b$ be two nonseparating two--sided circles on $N$ such that $N\bez a$ and $N\bez b$ are nonorientable. Then $t_a$ is conjugate in ${\cal{T}}(N)$ either to $t_b$ or to $t_b^{-1}$.
\end{lem}
\begin{proof}
By Lemma \ref{lem:cong:ess}, it is enough to prove that there exits a sequence of two--sided circles
$\lst{p}{k}$ such that $p_1=a$, $p_k=b$ and $I(p_i,p_{i+1})=1$ for $i=1,\ldots,k-1$. In other words, using
the terminology of \cite{Kork-non}, we have to prove that $a$ and $b$ are dually equivalent. For a
nonorientable surface with punctures this was proved in Theorem 3.1 of \cite{Kork-non}. It is
straightforward to check that the same proof applies to the case of a surface with boundary.
\end{proof}
\begin{lem}\label{lem:b:psi:cong}
Let $g=2r+2\geq 4$. Then $t_{b_{r+1}}$ is conjugate in ${\cal{T}}(N)$ to $t_\psi^{-1}$, where $\psi=y^{-1}(b_{r+1})$ is as in Theorem \ref{tw:gen:T}.
\end{lem}
\begin{proof}
Figure \ref{fig:115_bIPsi}(i) shows the circle $b_{r+1}$ as a circle on $\fal{N}$.
\begin{figure}[h]
\includegraphics{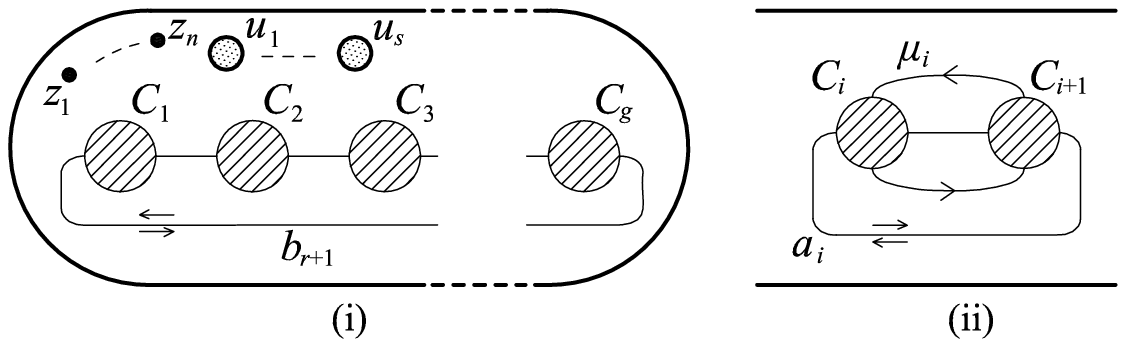}
\caption{Circle $b_{r+1}$ and crosscap slide $y_i$, Lemma \ref{lem:b:psi:cong}.}\label{fig:115_bIPsi}
\end{figure}
Using the notation from this figure, let $y_i$ for $i=1,\ldots, g-1$, be a slide of the crosscap $C_{i+1}$ along the loop $\mu_i$ as shown in Figure \ref{fig:115_bIPsi}(ii). In particular $y_{g-1}=y$.
It is straightforward to check that
\[y_{g-2}t_{a_{g-2}}y_{g-3}t_{a_{g-3}}\cdots y_2t_{a_2}y_1(b_{r+1})=y_{g-1}^{-1}(b_{r+1})=\psi. \]
Moreover, $y_{g-2}t_{a_{g-2}}\cdots y_2t_{a_2}y_1$, as a product of twists and even number of crosscap slides, is in the kernel of the determinant homomorphism $\map{D}{{\cal{PM}^+(N)}}{\zz_2}$ defined in Section \ref{sec:act}. Therefore, by Corollary \ref{wn:T:eq:falT}, $t_{b_{r+1}}$ is conjugate in ${\cal{T}}(N)$ to either $t_{\psi}$ or $t_{\psi}^{-1}$. Careful examination of the orientations of local neighbourhoods of $b_{r+1}$ and $\psi$ shows that in fact $t_{b_{r+1}}$ is conjugate to $t_\psi^{-1}$.
\end{proof}
\begin{lem} \label{lem:tw:sq:triv}
Let $g\geq 4$. Then $[t_{a_1}^2]=0$ in
$\mathrm{H_1}({\cal{T}}(N))$.
%
\end{lem}
\begin{proof}
Figure \ref{fig:120_hom_sqtw}(i) shows three two--sided nonseparating circles $a,b$ and $c$ such that $I(a,b)=1$, $I(b,c)=1$ and the complement of each of these circles in $N$ is nonorientable. Hence by Lemma \ref{lem:cong:ess}, all three twists $t_a,t_b$ and $t_c$ are conjugate in ${\cal{T}}(N)$. 
\begin{figure}[h]
\includegraphics{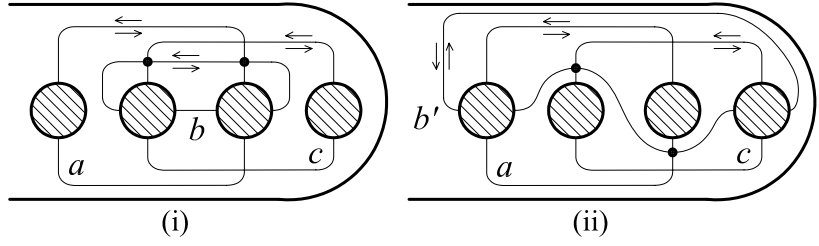}
\caption{Circles $a,b,b'$ and $c$, Lemma \ref{lem:tw:sq:triv}.}\label{fig:120_hom_sqtw}
\end{figure}
Similarly, Figure \ref{fig:120_hom_sqtw}(ii) shows that $t_a,t_{b'}$ and $t_{c}^{-1}$ are
conjugate. Hence $t_c$ and $t_c^{-1}$ are conjugate in ${\cal{T}}(N)$. By Lemma \ref{lem:two:gener:cong}, the same is true for $t_{a_1}$.
\end{proof}
\begin{uw}\label{uw:iness}
Observe that if $a$ and $b$ are two--sided nonseparating circles on $N$ such that $N\bez a$ and $N\bez b$ are nonorientable, then Lemma \ref{lem:two:gener:cong} gives us no hint if $t_a$ is conjugate to $t_b$ or maybe to $t_b^{-1}$. However, by Lemma \ref{lem:tw:sq:triv}, as long as $g\geq 4$ and we are concern with homology classes, the above ambiguity is inessential.
\end{uw}

Now let recall some results form \cite{Stukow_SurBg}.

\begin{lem}[Lemma 6.7 of \cite{Stukow_SurBg}] \label{lem:a_1:g7}
Let $g\geq 7$. Then $[t_{a_1}]=0$.\qed
\end{lem}
\begin{lem}[Lemma 6.6 of \cite{Stukow_SurBg}]\label{lem:brp1:g6}
Assume $g=2r+2\geq 6$. Then $[t_{b_{r+1}}]=0$.\qed
\end{lem}
\begin{lem}[Lemma 6.12 of \cite{Stukow_SurBg}] \label{lem:bdTwTriv_g5}
Assume $g\geq 5$. Then the boundary twists $t_{u_1},\ldots,t_{u_s}$ are trivial in
$\mathrm{H_1}({\cal{T}}(N))$.\qed
\end{lem}
\begin{lem}[Lemma 6.14 of \cite{Stukow_SurBg}]\label{lem:kapp}
Assume that $g\geq 3$ and let $\kappa$ be a two--sided separating circle on $N$ such that one of the components of $N\bez \kappa$ is a disk $\Delta$ containing the punctures $\lst{z}{n}$ and the holes $\lst{u}{s}$. Moreover, assume that the orientation of $\Delta$ agrees with the orientations of neighbourhoods of $\lst{u}{s}$ and $\kappa$. Then
\[\left[t_\kappa\right]=\left[t_{u_1}\cdots t_{u_s} \right]=\left[t_{u_1}\right]+\cdots +\left[t_{u_s}\right] .\]\qed
\end{lem}

Although each of the above lemmas was originally stated in terms of the group ${\cal{PM}^+(N_{g,s}^n)}$,
by Lemmas \ref{lem:cong:ess}, \ref{lem:tw:sq:triv} and by Remark \ref{uw:iness}, their proofs work as well in the case of the
twist subgroup ${\cal{T}}(N)$.
\begin{lem}\label{lem:triv:ksi}
Let $g\geq 4$ and let $\beta$ be a separating circle on $N$ such that one component of $N\bez\beta$ is Klein bottle with one hole and the second component is nonorientable. Then $[t_\beta]=0$. In particular \[[y^2]=[t_\xi]=0.\]
\end{lem}
\begin{proof}
Figure \ref{fig:130_hom_ksi} shows that there is a lantern configuration with one circle
$\beta$ and all other twists either trivial or conjugate to $t_{a_1}$. Hence
\[[t_{a_1}t_{a_1}]=[t_{\beta}t_{a_1}t_{a_1}].\]
\begin{figure}[h]
\includegraphics{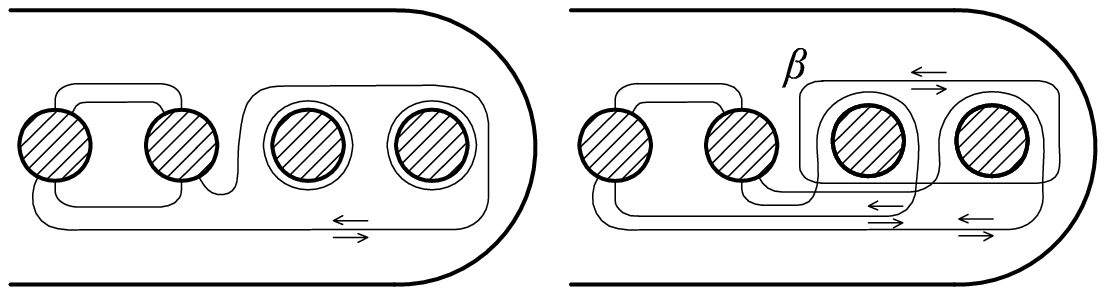}
\caption{Lantern relation, Lemma \ref{lem:triv:ksi}.}\label{fig:130_hom_ksi}
\end{figure}
\end{proof}
\begin{lem}\label{lem:triv:gamm}
Let $g\geq 4$ and let $\gamma$ be a circle on $N$ such that one of the components of $N\bez \gamma$ is a nonorientable surface of genus $3$ with one hole. Then $[t_\gamma]=0$.
\end{lem}
\begin{proof}
Figure \ref{fig:21_Homo} shows that there is a lantern configuration with one circle $\gamma$ and all other circles either bounding \Mob s or satisfying the assumptions of Lemma \ref{lem:triv:ksi}.
\begin{figure}[h]
\includegraphics{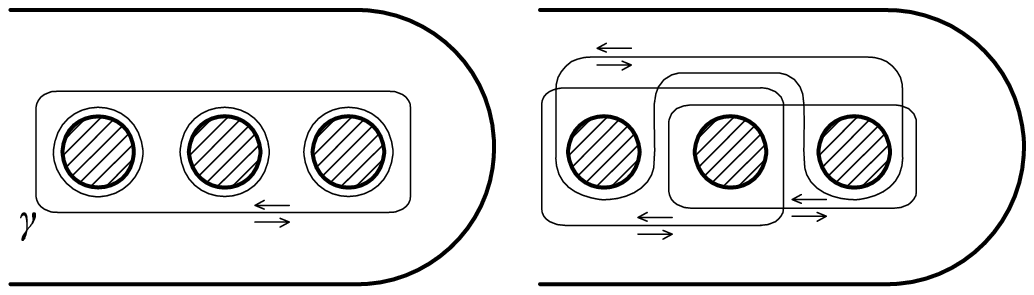}
\caption{Lantern relation, Lemma \ref{lem:triv:gamm}.}\label{fig:21_Homo}
\end{figure}
\end{proof}
\subsection{Homology classes of boundary twists}
\begin{lem}\label{lem:sum_u_i:g_4}
Let $g=4$. Then \[[t_{u_1}]+[t_{u_2}]+\cdots+[t_{u_s}]=0.\]
\end{lem}
\begin{proof}
By Lemma \ref{lem:kapp},
\[[t_{u_1}]+[t_{u_2}]+\cdots+[t_{u_s}]=t_\kappa,\]
where $\kappa$ is a circle on $N$ bounding all the punctures and boundary circles. On the other hand, Figure \ref{fig:21_Homo_2} shows that there is a lantern configuration with one circle $\gamma$ and all other circles either bounding \Mob s or satisfying the assumptions of Lemmas \ref{lem:triv:ksi} or \ref{lem:triv:gamm}.
\begin{figure}[h]
\includegraphics{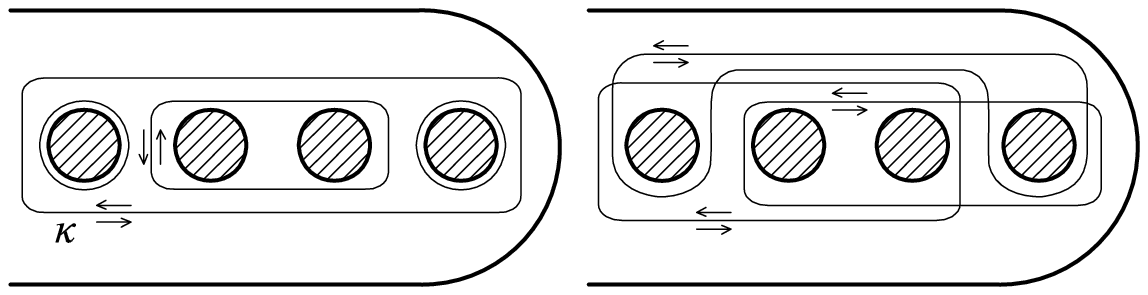}
\caption{Lantern relation, Lemma \ref{lem:sum_u_i:g_4}.}\label{fig:21_Homo_2}
\end{figure}

\end{proof}
\begin{lem}\label{lem:triv_u_j:sq}
Let $g\geq 3$. Then $2[t_{u_j}]=[t_{u_j}^2]=0$ for $j=1,\ldots,s$.
\end{lem}
\begin{proof}
Consider the lantern relation indicated in Figure \ref{fig:140_hom_u}.
\begin{figure}[h]
\includegraphics{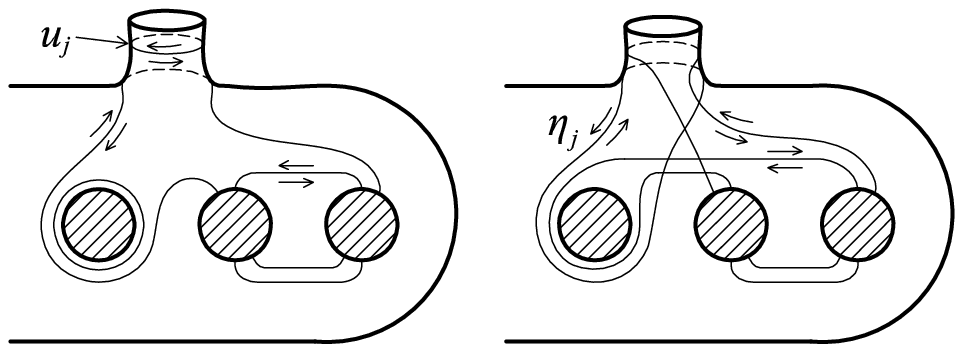}
\caption{Lantern relation $[t_{u_j}t_{a_1}t_{a_1}^{-1}]=[t_{\eta_j}t_{a_1}^{-1}t_{a_1}]$.} \label{fig:140_hom_u}
\end{figure}
Using Lemma \ref{lem:cong:ess}, it is easy to prove that all four twists about nonseparating circles in this figure are conjugate to $t_{a_1}$, hence we have
\[ [t_{u_j}t_{a_1}t_{a_1}^{-1}]=[t_{\eta_j}t_{a_1}^{-1}t_{a_1}]. \]
Therefore it is enough to show that $ft_{\eta_j}f^{-1}=t_{\eta_j}^{-1}$ for some  $f\in{\cal{T}}(N)$.

The circle $\eta_j$ divides the surface $N$ into a projective plane $N'$ with two holes and a nonorientable surface $N''$. Let $\map{h}{N}{N}$ be a diffeomorphism obtained as follows. On each of $N'$ and $N''$, $h$ is a slide of $\eta_j$ along the core of a crosscap such that $h$ is $-id$ on $\eta_j$. Clearly $h\in {\cal{PM}^+(N_{g,s}^n)}$ and $ht_{\eta_j}h^{-1}=t_{\eta_j}^{-1}$. By Proposition \ref{lem:ac:cross}, $D(y)=-1$ where $\map{D}{{\cal{PM}^+(N)}}{\zz_2}$ is the determinant homomorphism. Moreover $yt_{\eta_j}y^{-1}=t_{\eta_j}^{-1}$, hence by Corollary \ref{wn:T:eq:falT}, either $f=h$ or $f=hy$ is the required diffeomorphism.
\end{proof}
\subsection{Homology classes of twists for $g=3$}
\begin{lem}\label{lem:ess_g3_12}
Assume that $g=3$. 
Then
\[12[t_{a_1}]=[t_\xi]=[t_{u_1}]+[t_{u_2}]+\ldots+[t_{u_s}].\]
In particular $12[t_{a_1}]=[t_\xi]=0$ if $s=0$.
\end{lem}
\begin{proof}
Applying Lemma \ref{prop:tor:hole} to the configuration shown in Figure \ref{fig:160_hom_12t_a}, we
obtain
\[(t_{a_1}t_{a_2})^6=t_{\alpha}.\]
\begin{figure}[h]
\includegraphics{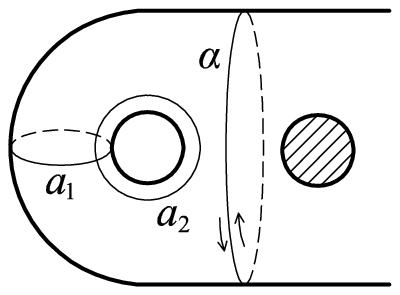}
\caption{Relation $(t_{a_1}t_{a_2})^6=t_{\alpha}$.} \label{fig:160_hom_12t_a}
\end{figure}
Hence by Lemmas \ref{lem:cong:ess} and \ref{lem:two:gener:cong},
\begin{equation} 12[t_{a_1}]=[t_{\alpha}].\label{eq:12taAlpha}\end{equation}
Using Lemma \ref{lem:cong:ess}, one can check that all twists about nonseparating essential circles indicated in Figure \ref{fig:170_hom_AlphKapp} are conjugate to $t_{a_1}$.
\begin{figure}[h]
\includegraphics{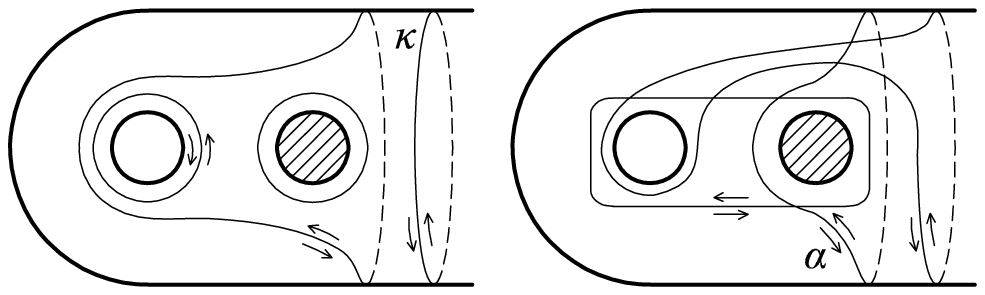}
\caption{Relation $[t_{a_1}t_{a_1}t_{\kappa}]=[t_{a_1}t_{a_1}t_\alpha]$.} \label{fig:170_hom_AlphKapp}
\end{figure}
Therefore we have a lantern relation
\[[t_{a_1}t_{a_1}t_{\kappa}]=[t_{a_1}t_{a_1}t_\alpha], \]
where $\kappa$ is as in Lemma \ref{lem:kapp}. Together with \eqref{eq:12taAlpha} this yields
\begin{equation}12[t_{a_1}]=[t_{\kappa}].\label{eq:t_at_kappa}\end{equation}

On the other hand, by the lantern relation provided by Figure \ref{fig:180_hom_KsiKapp}, we have
\[[t_\kappa t_{a_1} t_{a_1}^{-1}]=[t_\xi t_{a_1}^{-1}t_{a_1}]. \]
\begin{figure}[h]
\includegraphics{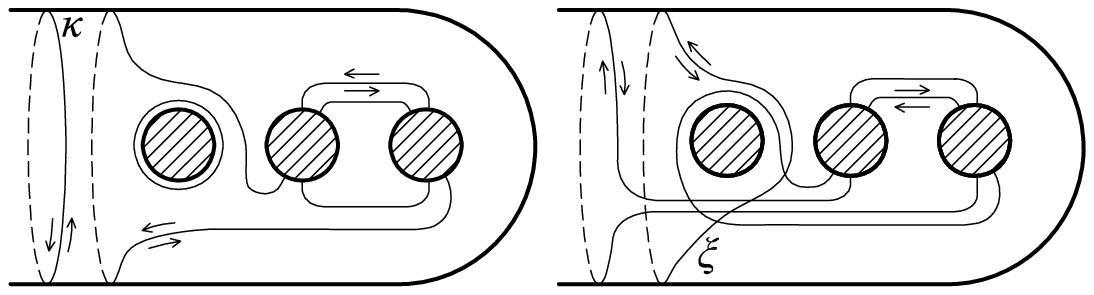}
\caption{Relation $[t_\kappa t_{a_1} t_{a_1}^{-1}]=[t_\xi t_{a_1}^{-1}t_{a_1}]$.} \label{fig:180_hom_KsiKapp}
\end{figure}
By \eqref{eq:t_at_kappa}, this gives $12[t_{a_1}]=[t_\kappa]=[t_\xi]$. Moreover, by Lemma \ref{lem:kapp},
\[[t_\kappa]=[t_{u_1}]+[t_{u_2}]+\ldots+[t_{u_s}]. \]
\end{proof}
\begin{lem}\label{lem:triv_g3_ta_24}
Assume that $g=3$. 
Then $24[t_{a_1}]=0$.
\end{lem}
\begin{proof}
Figure \ref{fig:150_hom_3ksi} shows that there is a lantern relation
\[t_{\kappa}=t_{\xi}t_{\xi_1}t_{\xi_2},\]
where $\kappa$ is as in Lemma \ref{lem:ess_g3_12}.
\begin{figure}[h]
\includegraphics{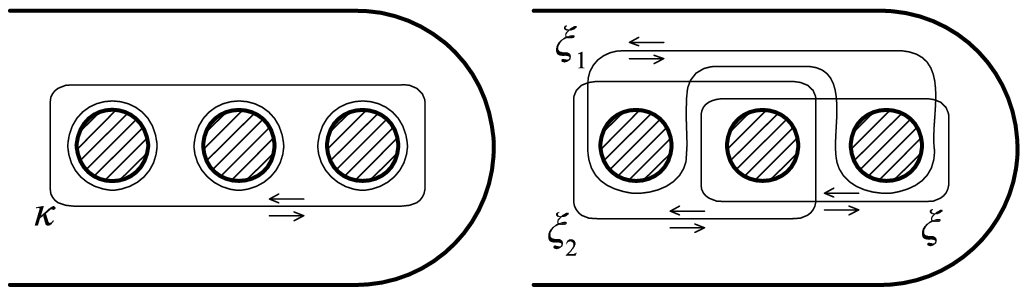}
\caption{Lantern relation $t_{\kappa}=t_{\xi}t_{\xi_1}t_{\xi_2}$.} \label{fig:150_hom_3ksi}
\end{figure}
Moreover by the proof of that lemma, $[t_{\xi}]=[t_{\xi_1}]=[t_{\xi_2}]=[t_\kappa]$. Hence
\[[t_\xi]=3[t_\xi].\]
Using once again Lemma \ref{lem:ess_g3_12}, we have $24[t_{a_1}]= 2[t_\xi]=0$.
\end{proof}
\subsection{Some special cases}
\begin{prop}\label{prop:g3_and_g4}
Let $N=N_{g,s}^n$ be a nonorientable surface of genus $g$ with $s$ holes and $n$ punctures. Then
\[\mathrm{H_1}({\cal{T}}(N),\zz)=
\begin{cases}
\gen{[t_{a_1}]}\cong \zz_{12} &\text{for $(g,s,n)=(3,0,0)$,}\\
\gen{[t_{a_1}]}\cong \zz_{24} &\text{for $(g,s,n)=(3,1,0)$,}\\
\gen{[t_{a_1}],[t_{b_{r+1}}]}\cong \zz_2\times \zz &\text{for $(g,s,n)=(4,0,0)$.}
\end{cases}
\]
\end{prop}
\begin{proof}
By Theorem 3 of \cite{BirChil1}, the group ${{\cal{M}(N_3)}}$ has a presentation
\begin{multline*}{{\cal{M}(N_3)}}=\gen{t_{a_1},t_{a_2},y\,\st\, t_{a_1}t_{a_2}t_{a_1}=t_{a_2}t_{a_1}t_{a_2},
yt_{a_1}y^{-1}=t_{a_1}^{-1},\\ yt_{a_2}y^{-1}=t_{a_2}^{-1}, y^2=1, (t_{a_1}t_{a_2})^6=1}.
\end{multline*}
Using $U=\{1,y\}$ as a transversal for the subgroup ${{\cal{T}}(N_3)}$, it is straightforward to obtain that
\[{{\cal{T}}(N_3)}=\gen{t_{a_1},t_{a_2}\,\st\, t_{a_1}t_{a_2}t_{a_1}=t_{a_2}t_{a_1}t_{a_2},(t_{a_1}t_{a_2})^6=1 }.\]
This implies that
\[\mathrm{H_1}({\cal{T}}(N_3))\cong\gen{t_{a_1}\,\st\,t_{a_1}^{12}=1}.\]
The reasoning for the surfaces $N_{3,1}$ and $N_4$ is similar, one has to use the known presentations for the
 groups ${\cal{M}(N_{3,1})}$ and ${\cal{M}(N_{4})}$  -- see Theorem 7.16 of \cite{Szep_curv} and Theorem 2.1 of \cite{Szep_gen4}.
\end{proof} 
\section{Computing $\mathrm{H_1}({\cal{T}}(N),\zz)$}
\begin{tw}\label{tw:main:hom}
Let $N=N_{g,s}^n$ be a nonorientable surface with $n$ punctures and $s$ holes. Then
\[
\mathrm{H_1}({\cal{T}}(N),\zz)=
\begin{cases}
 \zz_{12} & \text{if $g=3$ and $s=0$,}\\
 \zz_{24}\times \zz_2^{s-1} & \text{if $g=3$ and $s\geq 1$,}\\
 \zz_2\times \zz & \text{if $g=4$ and $s=0$,}\\
 \zz_2^{s}\times \zz & \text{if $g=4$ and $s\geq 1$,}\\
 \zz_2& \text{if $g=5,6$,}\\
 0&\text{if $g\geq 7$.}
\end{cases}
\]
\end{tw}
\begin{proof}
By Theorem \ref{tw:gen:T}, Corollary \ref{wn:T:eq:falT}, Lemmas \ref{lem:two:gener:cong} and \ref{lem:b:psi:cong},
$\mathrm{H_1}({\cal{T}}(N))$ is generated by
\begin{itemize}
 \item $[t_{a_1}],[t_\xi],[t_{u_1}],\ldots,[t_{u_s}]$ if $g$ is odd,
 \item $[t_{a_1}],[t_{b_{r+1}}],[t_\xi],[t_{u_1}],\ldots,[t_{u_s}]$ if $g$ is
 even.
\end{itemize}
Moreover
\begin{itemize}
 \item $[t_{a_1}]=0$ if $g\geq 7$ (Lemma \ref{lem:a_1:g7}),
 \item $[t_{b_{r+1}}]=0$ if $g\geq 6$ even (Lemma \ref{lem:brp1:g6}),
 \item $[t_{u_1}]=\cdots=[t_{u_s}]=0$ if $g\geq 5$ (Lemma \ref{lem:bdTwTriv_g5}),
 \item $[t_\xi]=0$ if $g\geq 4$ (Lemma \ref{lem:triv:ksi}),
 \item $[t_{u_1}]+\cdots+[t_{u_s}]=0$ if $g=4$ (Lemma \ref{lem:sum_u_i:g_4}),
 \item $[t_\xi]=[t_{u_1}]+\cdots+[t_{u_s}]=12[t_{a_1}]$ if $g=3$ (Lemma \ref{lem:ess_g3_12}).
\end{itemize}
Hence $\mathrm{H_1}({\cal{T}}(N))$ is generated by
\begin{itemize}
 \item $[t_{a_1}],[t_{u_1}],\ldots,[t_{u_{s-1}}]$ if $g=3$,
 \item $[t_{a_1}],[t_{b_{r+1}}],[t_{u_1}],\ldots,[t_{u_{s-1}}]$ if $g=4$,
 \item $[t_{a_1}]$ if $g=5,6$,
\end{itemize}
and $\mathrm{H_1}({\cal{T}}(N))=0$ if $g\geq 7$. In particular this concludes the proof in the case $g\geq
7$. Therefore in what follows we assume that $g\leq 6$.

We also know that
\begin{align}
 &2[t_{a_1}]=0 \quad \text{if $g\geq 4$ (Lemma \ref{lem:tw:sq:triv}),}\label{eq:rel1}\\
 &2[t_{u_1}]=\cdots =2[t_{u_{s-1}}]=0 \quad \text{if $g\geq 3$ (Lemma \ref{lem:triv_u_j:sq}),}\label{eq:rel2}\\
 &12[t_{a_1}]=0 \quad \text{if $g=3$ and $s=0$ (Lemma \ref{lem:ess_g3_12}),}\label{eq:rel4}\\
 &24[t_{a_1}]=0 \quad \text{if $g=3$ (Lemma \ref{lem:triv_g3_ta_24}).}\label{eq:rel3}
\end{align}
Hence it is enough to prove that every relation in the abelian group $\mathrm{H_1}({\cal{T}}(N))$ is a
consequence of the relations \eqref{eq:rel1}--\eqref{eq:rel3} above.

Let $g=3$ 
and suppose that
\begin{equation}\label{eq:1_tw}
\alpha [t_{a_1}]+\eps_1[t_{u_1}]+\cdots+\eps_{s-1}[t_{u_{s-1}}]=0.
\end{equation}
If $s>0$ define ${N'}$ to be a surface of genus $3$ obtained from $N$ by forgetting all the punctures and gluing a disk to each boundary component but one, say $u_s$. If $s=0$ define ${N'}$ by forgetting about the punctures in $N$.
We have a homomorphism
\[\map{\Phi}{{\cal{T}}(N)}{{\cal{T}}({N'})}.\]
Clearly $\Phi(t_{u_i})=0$ for $1\leq i \leq s-1$, hence equation \eqref{eq:1_tw} yields
\[\alpha[t_{a_1}]=0\quad\text{in $\mathrm{H_1}({\cal{T}}({N'}))$.} \]
By Proposition \ref{prop:g3_and_g4}, we have $24|\alpha$ if $s>0$ and $12|\alpha$ if $s=0$. Therefore by relations \eqref{eq:rel4} and \eqref{eq:rel3}, equation
\eqref{eq:1_tw} becomes
\begin{equation}\label{eq:2_tw}
\eps_1[t_{u_1}]+\cdots+\eps_{s-1}[t_{u_{s-1}}]=0.
\end{equation}
This concludes the proof if $s\leq 1$, hence assume that $s\geq 2$. Now let $\widehat{N}_j$ for
$j=1,\ldots,s-1$, be the surface obtained from $N$ by forgetting the punctures, gluing a cylinder to the
circles $u_j$ and $u_s$ and finally gluing a disk to each of the remaining boundary components. Then
$\widehat{N}_j$ is a closed nonorientable surface of genus $5$. Let
\[\map{\Upsilon_j}{{\cal{T}}(N)}{{\cal{M}}(\widehat{N}_j)}\]
be the homomorphism induced by inclusion. Since $\Upsilon_j(t_{u_i})=0$ for $i\neq j$ and $i\neq s$,
equation \eqref{eq:2_tw} gives us
\[\eps_j[t_{u_j}]=0\quad\text{in $\mathrm{H_1}({\cal{M}}(\widehat{N}_j))$}. \]
Since $u_j$ is a nonseparating two--sided circle on $\widehat{N}_j$, by Theorem 1.1 of \cite{Kork-non1}, we have
$2|\eps_j$. By relation \eqref{eq:rel2}, equation \eqref{eq:2_tw} becomes $0=0$, which completes the proof
for $g=3$.

The proof for $g=4$ is analogous. If we assume that
\[\alpha [t_{a_1}]+\beta [t_{b_{r+1}}]+\eps_1[t_{u_1}]+\cdots+\eps_{s-1}[t_{u_{s-1}}]=0\]
the we can show that $\alpha=\beta=0$ in the same manner as in the case $g=3$, namely by mapping $N$ into a closed surface of genus $4$ and using Proposition \ref{prop:g3_and_g4}. Similarly, we can show that $\eps_1=\ldots\eps_{s-1}=0$ by mapping $N$ into a closed surface of genus $6$ and using Theorem 1.1 of~\cite{Kork-non1}.

If $g=5$ or $g=6$ the proof is even simpler, it is enough to map $N$ into a closed surface and use Theorem 1.1 of \cite{Kork-non1}, we skip the details.
\end{proof}

\bibliographystyle{abbrv}
\bibliography{mybib}
\end{document}